\documentclass[reqno]{amsart}
\usepackage{amssymb,epsfig,graphics,graphicx,bbm,hyperref,color}
\usepackage{xfrac,mathtools,MnSymbol}

\usepackage{multicol}
\usepackage{enumitem}

\definecolor{gray}{rgb}{0.93,0.93,0.93}
\definecolor{light-gold}{rgb}{0.99,0.97,0.78}

\def\be{\begin{equation}}
\def\ee{\end{equation}}
\def\bm{\begin{multline}}
\def\eem{\end{multline}}
\def\bfig{\begin{figure}[htb]}
\def\efig{\end{figure}}

\newcommand{\e}[1]{\,{\rm e}^{#1}\,}

\numberwithin{equation}{section}
\newtheorem{theorem}{Theorem}[section]
\newtheorem{proposition}[theorem]{Proposition}
\newtheorem{lemma}[theorem]{Lemma}

\newcommand{\bbE}{{\mathbb E}}

\newcommand{\bbP}{{\mathbb P}}

\newcommand{\bbZ}{{\mathbb Z}}

\newcommand{\eps}{{\varepsilon}}

\newcommand{\EE}{\mathbb{E}}

\newcommand{\PP}{\mathbb{P}}

\newcommand{\RR}{\mathbb{R}}

\newcommand{\cL}{\mathcal{L}}
\newcommand{\cE}{\mathcal{E}}
\newcommand{\cM}{\mathcal{M}}
\newcommand{\cS}{\mathcal{S}}
\renewcommand{\a}{\alpha}
\renewcommand{\b}{\beta}
 
\renewcommand{\d}{\delta}

\renewcommand{\L}{\Lambda}
\newcommand{\om}{\omega}
\newcommand{\Om}{\Omega}
 
\newcommand{\s}{\sigma}

\newcommand{\es}{\varnothing}
\newcommand{\se}{\subseteq}
\newcommand{\sm}{\setminus}
\newcommand{\lra}{\leftrightarrow}
\newcommand{\crit}{\mathrm{c}}
\newcommand{\oo}{\infty}
\newcommand{\one}{\hbox{\rm 1\kern-.27em I}}

\newcommand{\fa}{\mathfrak{a}}
\newcommand{\fb}{\mathfrak{b}}

\makeatletter
  \def\tagform@#1{\maketag@@@{\footnotesize{(#1)}\@@italiccorr}}
\makeatother

\renewcommand{\eqref}[1]{(\ref{#1})}

\begin{document}


\title{Critical parameter of random loop model on trees}

\author{Jakob E. Bj\"ornberg}
\address{Department of Mathematics,
University of Gothenburg, 
Sweden}
\email{jakob.bjornberg@gmail.com}

\author{Daniel Ueltschi}
\address{Department of Mathematics, University of Warwick,
Coventry, CV4 7AL, United Kingdom}
\email{daniel@ueltschi.org}



\begin{abstract} 
We give estimates of the critical parameter for
random loop models that are related to quantum spin systems.
A special case of the model that we consider is the interchange- or
random-stirring process.
We
consider here the model defined on regular trees of large degrees,
which are expected to approximate high
spatial dimensions. We find a critical parameter that indeed shares
similarity with existing numerical results for the cubic lattice.
In the case of the interchange process our results improve on earlier
work by Angel and by Hammond, in that we determine the second-order
term of the critical parameter.
\end{abstract}

\thanks{\copyright{} 2017 by the authors. This paper may be reproduced, in its
entirety, for non-commercial purposes.}

\maketitle

\section{Introduction}
We consider random loop models that are motivated by quantum spin
systems. A special case is the random interchange model that was first
introduced by Harris \cite{Har}. T\'oth showed that a variant of this
model, where permutations receive the weight $2^{\rm \#\, cycles}$, is
closely related to the quantum Heisenberg
ferromagnet \cite{Toth}. Another loop model was introduced by Aizenman
and Nachtergaele to describe the quantum Heisenberg
antiferromagnet \cite{AN}. These loop models were combined in order to
describe a family of quantum systems that interpolate between the two
Heisenberg models, and which contains the quantum XY model \cite{Uel}.

Let $G = (V,E)$ be an arbitrary finite graph with vertex set $V$ and
edge set $E$, and $\beta>0$, $u \in [0,1]$ be two parameters. To each
edge $e \in E$ is assigned a time interval $[0,\beta]$, and an
independent Poisson point process with two kinds of outcomes:
``crosses" occur with intensity $u$ and ``double bars" occur with
intensity $1-u$. We let $\Omega(G)$ denote the set of realizations of
the combined Poisson point process on $E \times [0,\beta]$.

Given a realization $\omega \in \Omega(G)$, we consider the loop
passing through a point $(x,t) \in V \times [0,\beta]$ that is defined
in a natural way, as follows
(see Fig.\ \ref{fig loops}). The loop is a closed
trajectory with support on $V \times [0,\beta]_{\rm per}$ where
$[0,\beta]_{\rm per}$ is the interval $[0,\beta]$ with periodic
boundary conditions, i.e., the torus of length $\beta$. Starting at
$(x,t)$, move ``up" until meeting the first cross or double bar with
endpoint $x$; then jump onto the other endpoint, and continue in the
same direction if a cross, in the opposite direction if a double bar;
repeat until the trajectory returns to $(x,t)$.

\begin{centering}
\bfig
\begin{picture}(0,0)%
\epsfig{file=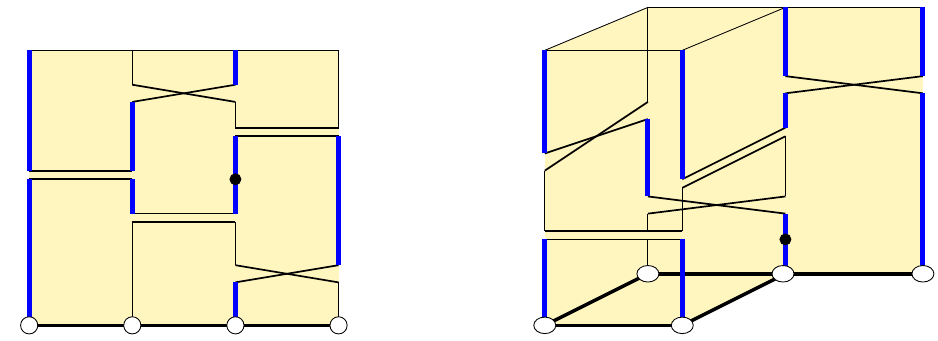}%
\end{picture}%
\setlength{\unitlength}{2171sp}%
\begingroup\makeatletter\ifx\SetFigFont\undefined%
\gdef\SetFigFont#1#2#3#4#5{%
  \reset@font\fontsize{#1}{#2pt}%
  \fontfamily{#3}\fontseries{#4}\fontshape{#5}%
  \selectfont}%
\fi\endgroup%
\begin{picture}(8159,2948)(1246,-4090)
\put(2671,-2761){\makebox(0,0)[lb]{\smash{{\SetFigFont{7}{8.4}{\rmdefault}{\mddefault}{\updefault}{\color[rgb]{0,0,0}$(x,t)$}%
}}}}
\put(7876,-3961){\makebox(0,0)[lb]{\smash{{\SetFigFont{8}{9.6}{\rmdefault}{\mddefault}{\updefault}{\color[rgb]{0,0,0}$G$}%
}}}}
\put(1276,-1636){\makebox(0,0)[lb]{\smash{{\SetFigFont{8}{9.6}{\rmdefault}{\mddefault}{\updefault}{\color[rgb]{0,0,0}$\beta$}%
}}}}
\put(5776,-1636){\makebox(0,0)[lb]{\smash{{\SetFigFont{8}{9.6}{\rmdefault}{\mddefault}{\updefault}{\color[rgb]{0,0,0}$\beta$}%
}}}}
\put(1246,-4036){\makebox(0,0)[lb]{\smash{{\SetFigFont{8}{9.6}{\rmdefault}{\mddefault}{\updefault}{\color[rgb]{0,0,0}$0$}%
}}}}
\put(5731,-4036){\makebox(0,0)[lb]{\smash{{\SetFigFont{8}{9.6}{\rmdefault}{\mddefault}{\updefault}{\color[rgb]{0,0,0}$0$}%
}}}}
\put(4351,-4036){\makebox(0,0)[lb]{\smash{{\SetFigFont{8}{9.6}{\rmdefault}{\mddefault}{\updefault}{\color[rgb]{0,0,0}$G$}%
}}}}
\put(8176,-3286){\makebox(0,0)[lb]{\smash{{\SetFigFont{7}{8.4}{\rmdefault}{\mddefault}{\updefault}{\color[rgb]{0,0,0}$(x,t)$}%
}}}}
\end{picture}%
\caption{Graphs and realizations of Poisson point processes, and the loop that contains $(x,t)$.}
\label{fig loops}
\efig
\end{centering}

In order to represent a quantum model, one should add the weight
$\theta^{\rm \#\, loops}$ with $\theta = 2, 3, 4, \dots$; quantum
correlations are then given in terms of loop correlations, and
magnetic long-range order is equivalent to the presence of macroscopic
loops. Notice that the parameter $\beta$ plays the r\^ole of the
inverse temperature of the quantum spin system, hence the notation.

The random interchange model (i.e.\ the case $u=1$ and $\theta=1$) has
been the object of several studies when the graph is a tree
\cite{Ang,Ham1, Ham2},
the complete graph \cite{Sch, Ber, BK}, the
hypercube \cite{KMU}, and the Hamming graph \cite{MS};
a result for general graphs was also proved
in \cite{AK}. In the case of arbitrary $\theta\in\{2,3,\dotsc\}$, 
and on the complete
graph, the critical parameter has also been determined
\cite{Bjo,Bjo2}. 
Another generalization of
the random interchange model is Mallows permutations, studied
in \cite{GP,Sta}.

The occurrence of macroscopic loops can be proved using the method of
reflection positivity and infrared bounds in the case where $u \in
[0,\frac12]$, $\theta=2,3,...$, and a cubic lattice of sufficiently
high dimensions (depending on $\theta$); see \cite{Uel} for precise
statements.

In the case where the graph is a three-dimensional cubic lattice with
edges between nearest-neighbors, and with $\theta=1$, the critical
parameter $\beta_{\rm c}(u)$ has been calculated numerically
in \cite{BBBU}. The result is depicted in Fig.\ \ref{fig betacrit} and
shows a convex curve where $\beta_{\rm c}(0)$ is slightly smaller than
$\beta_{\rm c}(1)$ and which has a minimum at or around $u=\frac12$. This behavior is expected to hold for all dimensions $d \geq 3$.

\begin{centering}
\bfig
\includegraphics[width=54mm]{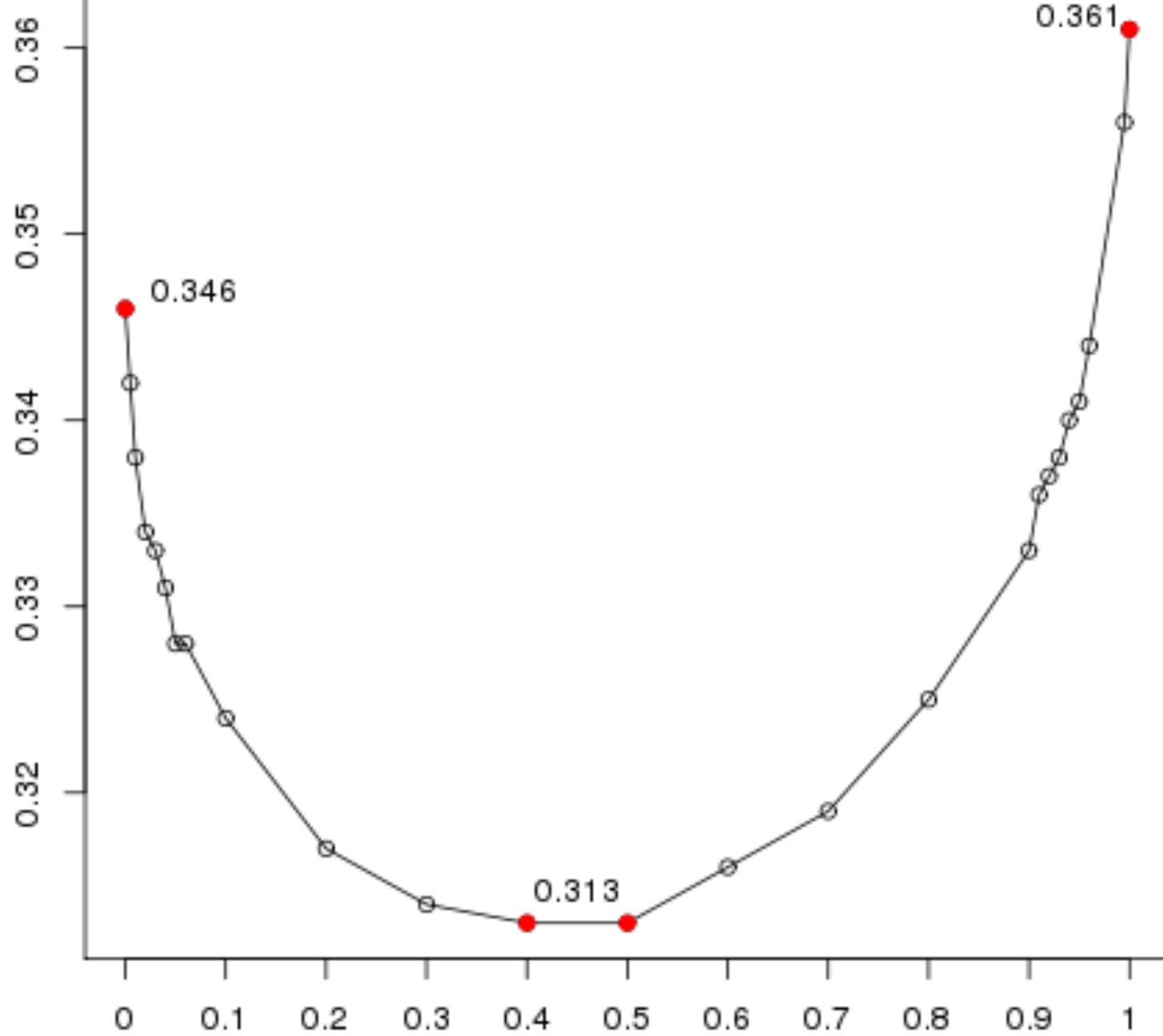}%
\qquad\qquad
\raisebox{8mm}{\includegraphics[width=55mm]{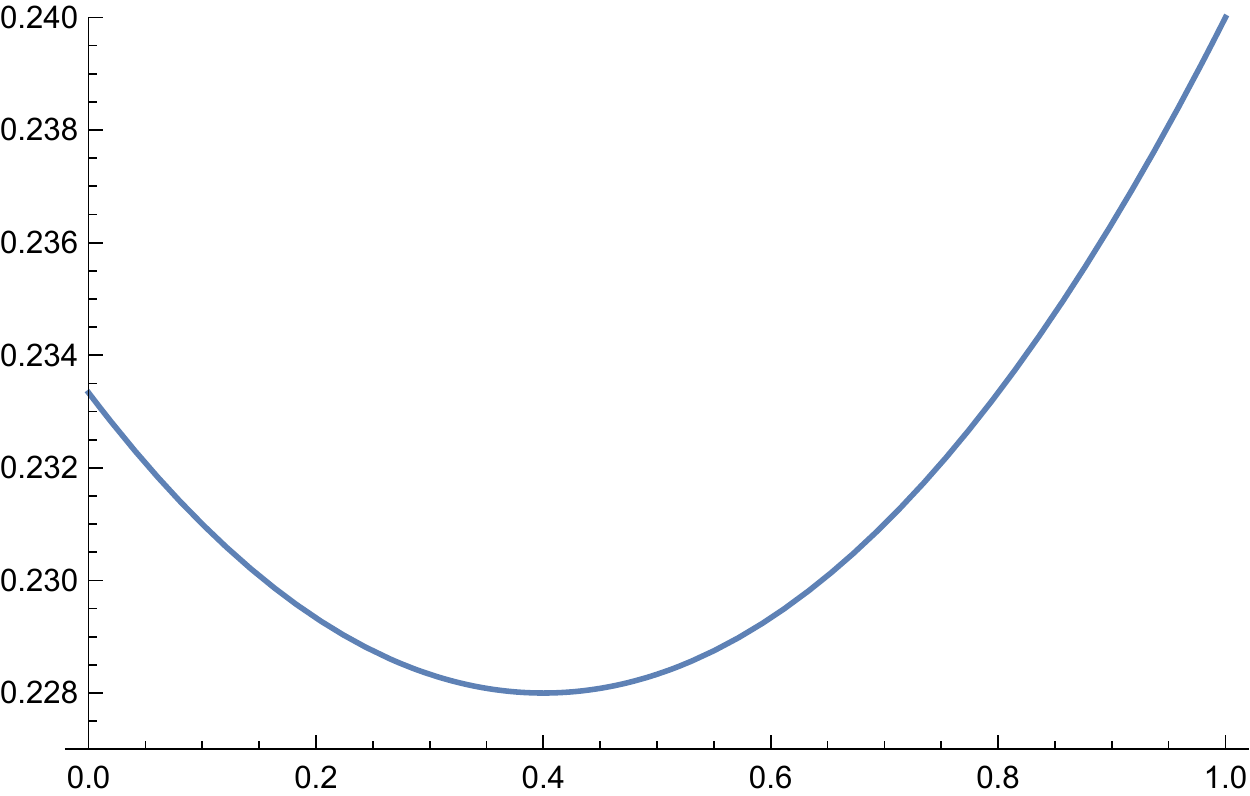}}
\caption{The critical parameter $\beta_{\rm c}$ as function of $u$: Left, on the three-dimensional cubic lattice (numerical results from \cite{BBBU}); right, Eq.\ \eqref{beta crit} with $d=5$.}
\label{fig betacrit}
\efig
\end{centering}

Trees are expected to approximate high dimensions. We consider here
infinite regular trees with offspring degree $d$. 
Loops are almost-surely well-defined in the same way as previously
(since vertices have uniformly bounded degrees) but now some loops may
be unbounded.
We prove that a transition takes
place at the critical parameter $\beta_{\rm c} = \beta_{\rm c}(u,d)$
given by
\be
\label{beta crit}
\beta_{\rm c}(u,d) = \frac1d + \frac{1 - u(1-u)-\tfrac16(1-u)^2}{d^2} 
+ o(d^{-2}).
\ee
The second graph of Fig.\ \ref{fig betacrit} shows $\beta_{\rm c}$ as function of $u$ (with $d=5$). The leading order of the critical parameter, $\frac1d$, is also the leading order for the percolation threshold in the associated percolation model where an edge is open
if at least one cross or double bar is present in the corresponding
interval $[0,\beta]$. The next order for $\beta_{\rm c}$ is a non-trivial function of
$u$ and it is smallest for $u=\frac25$. This function can be understood by looking at edges with two links and at loop connections with two crosses, two double bars, or one each. As is explained below (see Fig.\ \ref{fig two-connection}), loop connections are better in the latter
case, with a cross and a double bar.\footnote{Alan Hammond pointed out to us this important
observation.}

Let $E_\infty$ denote the event where the root of the tree (at time 0) 
belongs to an infinite loop.

\begin{theorem}
\label{thm beta crit}
Let $A>0$ be arbitrary and $\beta = \frac1d + \frac\alpha{d^2}$ 
with $\alpha \leq A$. 
There exists $d_0$ (that may depend on $A$ but 
not on $\alpha$) such that for all $d \geq d_0$, 
there exists $\alpha_\crit(u,d)$ such that
\[
\bbP_{\beta,d,u}(E_\infty) \begin{cases} = 0 & \text{if } \alpha < \alpha_{\rm c}(u,d), \\ >0 & \text{if } \alpha > \alpha_{\rm c}(u,d). \end{cases}
\]
Further, we have 
$\alpha_\crit(u,d) = 1 - u(1-u)-\frac16(1-u)^2 +o(1)$
as $d\to\oo$, uniformly in $u$.
\end{theorem}

This theorem follows from Propositions \ref{betacrit-prop} and \ref{sharp-prop}. Proposition \ref{betacrit-prop} establishes the existence of $d_0(\alpha)$ such that loops occur for $\alpha>\alpha_{\rm c}$ but not for $\alpha<\alpha_{\rm c}$, if $d > d_0(\alpha)$. The case $u=1$, that is, the interchange model on trees, was treated up to first order in $d^{-1}$ by Hammond \cite{Ham1}, following the work of Angel \cite{Ang}.
Proposition \ref{sharp-prop} implies that $d_0(\alpha)$ is uniform on
bounded intervals. The corresponding result for the interchange model
($u=1$) was proved by Hammond \cite{Ham2}. 
It turns out that his method can be adapted to $u\neq1$ with minor modifications, as explained in Section \ref{sec sharp}.

The reason we require $\alpha$ to be bounded by $A$ is that our
arguments only apply for $\beta$ close to $d^{-1}$;  hence $d_0$ depends
on $A$.  Presumably there is some $A_0>0$ such that $\PP(E_\oo)>0$
whenever $\alpha>A_0$.  For the interchange model this was also proved
by Hammond \cite{Ham1}, but his arguments use in a crucial way a
comparison with random walk, which fails for $u<1$ due to the
`time-reversal' which occurs when a double bar is traversed.  (For
$\alpha<0$ we have $\PP(E_\oo)=0$ for all $d$ by standard comparison
with percolation.)

Of the two previous methods for proving the occurrence of infinite
loops in the interchange model, due to Angel \cite{Ang} and Hammond
\cite{Ham1} respectively, our argument is thus closer to that of
Angel, which also 
requires $\beta$ to be close to $d^{-1}$.  However, where Angel uses a
comparison with a branching process, we instead directly prove
recursion inequalities for the probability of long loops.  These
inequalities include error terms which are of higher order in $d^{-1}$ and
may be made negligible by taking $d$ large.

The case $\theta\neq1$ could probably be treated in a similar way,
although a full study is needed in order to rule out extra
obstacles. A major open problem is to establish
that, in the case where the graph is a box in $\bbZ^{d'}$ with
nearest-neighbor edges, the critical parameter satisfies
Eq.\ \eqref{beta crit} with $d = 2d'-1$.

\section{The critical parameter}

As mentioned above, we consider an infinite rooted regular tree with
offspring degree $d$.  To
each edge is associated the interval $[0,1]$, and an independent
Poisson point process where ``crosses" occur with intensity 
$u\beta \in [0,\beta]$, and ``double bars" occur with 
intensity $(1-u)\beta$. (This is a variation of the model discussed above, with $\beta$ affecting the intensities rather than the time interval, which is obviously equivalent.)

Let us define $\bar\alpha(u) = 1 - u(1-u)-\frac16(1-u)^2$.
In what follows we always take $0\leq\alpha\leq A$ for some arbitrary
but fixed $A>0$, and error terms may depend on $A$.

\begin{proposition}\label{betacrit-prop}
Let $\b=\tfrac1d+\frac\alpha{d^2}$ and $\d>0$.
There exists $d_0(\d)$ such that the following hold for all $d>d_0$.
\begin{enumerate}[leftmargin=*]
\item[(a)] For every $\a\leq\bar\alpha-\d$, 
we have \[\bbP_{\beta,d,u}((\rho,0) \leftrightarrow \infty) = 0.\]
\item[(b)] For every $\a\geq\bar\alpha+\d$, we 
have \[\bbP_{\beta,d,u}((\rho,0) \leftrightarrow \infty)  > 0.\]
\end{enumerate}
\end{proposition}

Note that we prove exponential decay for (a), that is, the loop containing $(\rho,0)$ has diameter $m$ with probability less than $C \e{-\eta m}$.
These claims can be compared to the numerical results for
three-dimensional lattices.  Also, the special case 
$u=1$ of our result gives a
solution to Problem~10 of~\cite{Ang} (for large enough $d$).

\subsection{Preliminaries}

We let $T$ denote an infinite tree where each vertex has $d\geq2$
offspring, and write $\rho$ for its root.  
For 
$m\geq 0$ let $T^{(m)}$ denote the subtree of $T$ consisting of
the first $m$ generations.

We write $\s_m$ for the probability that $(\rho,0)$ belongs to a
loop which reaches generation $m$ in $T^{(m)}$, and $\zeta_m=1-\s_m$. 
Note that $\s_m\leq\s_{m-1}$
and that $\s_m\to\PP((\rho,0)\lra\oo)$ as $m\to\oo$.
We write $B^m_{(\rho,0)}$ for the event that $(\rho,0)$ 
does not belong to a
loop which reaches generation $m$ in $T^{(m)}$, so that
$\PP(B^m_{(\rho,0)})=\zeta_m$.  Thus $B^m_{(\rho,0)}$ is the event
that the loop of $(\rho,0)$ is `blocked' from generation $m$, and
$\s_m$ is the probability that it `survives' for $m$ generations.

Crosses and double-bars will be referred to collectively as
\emph{links}.  If $(xy,t)\in\om$ is a link, then in general we have
that the points $(x,t+)$ and $(x,t-)$ may belong to different loops 
(the same is true for $(y,t+)$ and $(y,t-)$).  We say that a link is a
\emph{monolink} if $(x,t+)$ and $(x,t-)$ belong to the \emph{same}
loop.  The following simple observation will be useful.

\begin{proposition}\label{monolink-prop}
Suppose that $y$ is a child of $x$ in $T^{(m)}$.  If there is
only one link between $x$ and $y$ then it is a monolink.
\end{proposition}
\begin{proof}
Denote the link $(xy,t)$.  In the configuration obtained by removing
this link, the points $(x,t)$ and $(y,t)$ belong to two different
loops, since we are on a tree.  When the link is added back in, the
loops are merged to a single loop, since the tree is finite.  This
proves the claim. 
\end{proof}

Write $A_1$ for the event that, for each child $x$ of $\rho$, there is
at most one link between $\rho$ and $x$.  
Write $A_2$ for the event that:  (i) there is a unique child $x$ of
$\rho$ with exactly 2 links between $\rho$ and $x$, (ii) for all
siblings $x'$ of $x$ there is at most one link between $\rho$ and
$x'$, and (iii) for all children $y$ of $x$ there is at most one link
between $x$ and $y$.  See Fig.~\ref{fig events}.

\begin{centering}
\bfig
\begin{picture}(0,0)%
\includegraphics{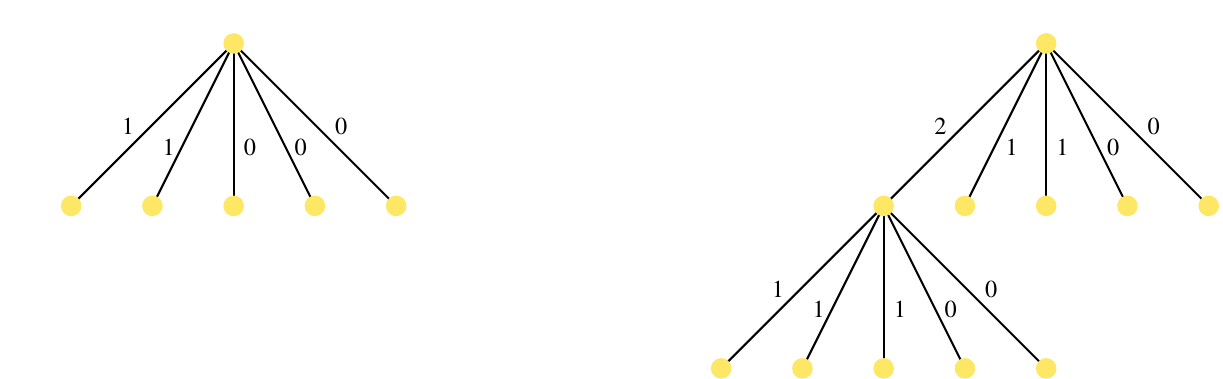}%
\end{picture}%
\setlength{\unitlength}{2565sp}%
\begingroup\makeatletter\ifx\SetFigFont\undefined%
\gdef\SetFigFont#1#2#3#4#5{%
  \reset@font\fontsize{#1}{#2pt}%
  \fontfamily{#3}\fontseries{#4}\fontshape{#5}%
  \selectfont}%
\fi\endgroup%
\begin{picture}(9001,2781)(676,-2836)
\put(5551,-586){\makebox(0,0)[lb]{\smash{{\SetFigFont{12}{14.4}{\rmdefault}{\mddefault}{\updefault}{\color[rgb]{0,0,0}$A_2$:}%
}}}}
\put(7051,-1486){\makebox(0,0)[lb]{\smash{{\SetFigFont{10}{12.0}{\rmdefault}{\mddefault}{\updefault}{\color[rgb]{0,0,0}$x$}%
}}}}
\put(676,-586){\makebox(0,0)[lb]{\smash{{\SetFigFont{12}{14.4}{\rmdefault}{\mddefault}{\updefault}{\color[rgb]{0,0,0}$A_1$:}%
}}}}
\put(2251,-211){\makebox(0,0)[lb]{\smash{{\SetFigFont{10}{12.0}{\rmdefault}{\mddefault}{\updefault}{\color[rgb]{0,0,0}$\rho$}%
}}}}
\put(8251,-211){\makebox(0,0)[lb]{\smash{{\SetFigFont{10}{12.0}{\rmdefault}{\mddefault}{\updefault}{\color[rgb]{0,0,0}$\rho$}%
}}}}
\end{picture}%
\caption{Summing over two events $A_1$ and $A_2$.}
\label{fig events}
\efig
\end{centering}

\noindent
Clearly we have that
\be\label{zeta-split-eq}
\zeta_m=\PP(B^m_{(\rho,0)})=
\PP(B^m_{(\rho,0)}\cap A_1)
+\PP(B^m_{(\rho,0)}\cap A_2)
+\PP(B^m_{(\rho,0)}\sm (A_1\cup A_2)).
\ee
In the rest of
this section we work with $\beta$ of the form
\be
\beta=\tfrac1d+\tfrac\alpha{d^2}
\ee
for $\alpha\in\RR$.

\subsection{Occurrence of long loops}

We now prove part (b) of Proposition \ref{betacrit-prop}.
For given $m\geq1$ and $\eps>0$ we define 
\be
\tilde \s_m = \s_m\wedge \s_{m-1}\wedge (\tfrac\eps d)=
\s_{m}\wedge (\tfrac\eps d).
\ee
Recall that we assume $\alpha\leq A$.
In this section we show the following:
\begin{proposition}\label{large-prop}
For all $m\geq1$ we have 
\[
\s_m\geq \tilde \s_{m-1}+
\tfrac{\tilde \s_{m-1}}{d}(\alpha-\bar\alpha(u))-\tfrac12\tilde\s_{m-1}^2
+ O(d^{-3}),
\]
where the $O(d^{-3})$ is uniform in $m$ (but depends on $A$).
\end{proposition}
Given the proposition, we can establish the occurrence of infinite
loops:
\begin{proof}[Proof of Proposition~\ref{betacrit-prop}, part (b)]
We claim that if $\eps<2(\alpha-\bar\alpha(u))$ is small enough then,
for $d$ large enough, we have $\s_m\geq\tilde\s_m\geq\tfrac\eps d$
for all $m$.
Since $\s_0=1$ and $\s_1\geq 1- (e^{-\beta})^d$, the claim holds for
$m=1$, and~Prop.\ref{large-prop} gives the claim by induction.
 Hence $\s_m\geq\tfrac\eps d$ for all $m$, which
gives the result.
\end{proof}

\begin{proof}[Proof of Prop.~\ref{large-prop}]
The starting point is the inequality
\be\label{zeta-ub}
\zeta_m\leq 
\PP(B^m_{(\rho,0)}\cap A_1)
+\PP(B^m_{(\rho,0)}\cap A_2)
+1-\PP(A_1)-\PP(A_2),
\ee
which follows directly from~\eqref{zeta-split-eq}.
First note that
\be\label{PA12}
\PP(A_1)=(e^{-\beta}(1+\beta))^d,\qquad
\PP(A_2)=\tfrac12 d\beta^2e^{-\beta}(e^{-\beta}(1+\beta))^{2d-1}.
\ee
Next note that
\be\label{PA1}
\begin{split}
\PP(B^m_{(\rho,0)}\cap A_1)&=
\sum_{k=0}^d \binom{d}{k} (\b e^{-\b})^k(e^{-\b})^{d-k} 
(\zeta_{m-1})^k\\
&=\big(e^{-\b}(1+\b\zeta_{m-1})\big)^d.
\end{split}
\ee
This relies on Prop.~\ref{monolink-prop}.  Indeed, if there are $k$
children $x_1,\dotsc,x_k$  
of $\rho$ that are linked to $\rho$, with one link each, at
times $t_1,\dotsc,t_k$ say, then $(\rho,0)$ lies in the same loop as
all of $(x_1,t_1),\dotsc,(x_k,t_k)$.  The probability of not being
connected to generation 0 is the same if one has one incoming
link from a parent as if one has none, and is thus $\zeta_{m-1}$ for
each of $(x_1,t_1),\dotsc,(x_k,t_k)$.

In obtaining a similar expression for the case $A_2$, it is useful to
refer to Fig.~\ref{fig two-connection}.  
Let $\L_\rho$ and $\L_x$ denote the restrictions of the subset
highlighted in blue to $\{\rho\}\times[0,1]$ and 
$\{x\}\times[0,1]$, respectively.
Thus $\L_\rho$ and $\L_x$ have respective lengths 
$X$ and $1-X$ in the case of two crosses;
$X$ and $X$ in the case of two double-bars;  and 1 in the case of a
mixture.  It may look obvious that $X$ is uniformly distributed in $[0,1]$; this is however incorrect, since it can be written as 
\[
X=\min\{U_1,U_2\}+1-\max\{U_1,U_2\},
\]
where $U_1,U_2$ are independent uniform random variables on $[0,1]$;  
in particular $\EE[X]=\tfrac23$.

\bfig
\begin{picture}(0,0)%
\includegraphics{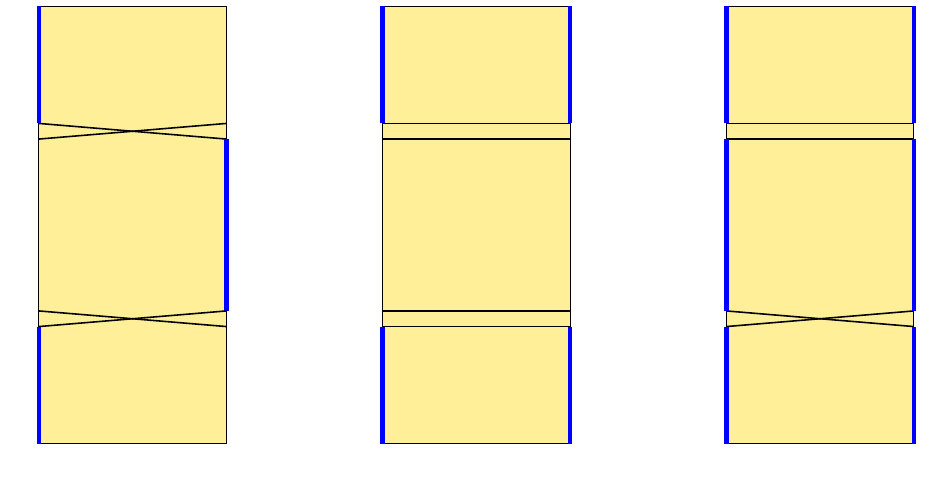}
\end{picture}%
\setlength{\unitlength}{1973sp}%
\begingroup\makeatletter\ifx\SetFigFont\undefined%
\gdef\SetFigFont#1#2#3#4#5{%
  \reset@font\fontsize{#1}{#2pt}%
  \fontfamily{#3}\fontseries{#4}\fontshape{#5}%
  \selectfont}%
\fi\endgroup%
\begin{picture}(9115,4616)(826,-4933)
\put(1051,-4861){\makebox(0,0)[lb]{\smash{{\SetFigFont{10}{12.0}{\rmdefault}{\mddefault}{\updefault}{\color[rgb]{0,0,0}$\rho$}%
}}}}
\put(826,-3436){\makebox(0,0)[lb]{\smash{{\SetFigFont{10}{12.0}{\rmdefault}{\mddefault}{\updefault}{\color[rgb]{0,0,0}$t_1$}%
}}}}
\put(826,-1636){\makebox(0,0)[lb]{\smash{{\SetFigFont{10}{12.0}{\rmdefault}{\mddefault}{\updefault}{\color[rgb]{0,0,0}$t_2$}%
}}}}
\put(4126,-3436){\makebox(0,0)[lb]{\smash{{\SetFigFont{10}{12.0}{\rmdefault}{\mddefault}{\updefault}{\color[rgb]{0,0,0}$t_1$}%
}}}}
\put(7426,-3436){\makebox(0,0)[lb]{\smash{{\SetFigFont{10}{12.0}{\rmdefault}{\mddefault}{\updefault}{\color[rgb]{0,0,0}$t_1$}%
}}}}
\put(4201,-1636){\makebox(0,0)[lb]{\smash{{\SetFigFont{10}{12.0}{\rmdefault}{\mddefault}{\updefault}{\color[rgb]{0,0,0}$t_2$}%
}}}}
\put(7426,-1636){\makebox(0,0)[lb]{\smash{{\SetFigFont{10}{12.0}{\rmdefault}{\mddefault}{\updefault}{\color[rgb]{0,0,0}$t_2$}%
}}}}
\put(2926,-4861){\makebox(0,0)[lb]{\smash{{\SetFigFont{10}{12.0}{\rmdefault}{\mddefault}{\updefault}{\color[rgb]{0,0,0}$x$}%
}}}}
\put(4351,-4861){\makebox(0,0)[lb]{\smash{{\SetFigFont{10}{12.0}{\rmdefault}{\mddefault}{\updefault}{\color[rgb]{0,0,0}$\rho$}%
}}}}
\put(7651,-4861){\makebox(0,0)[lb]{\smash{{\SetFigFont{10}{12.0}{\rmdefault}{\mddefault}{\updefault}{\color[rgb]{0,0,0}$\rho$}%
}}}}
\put(6226,-4861){\makebox(0,0)[lb]{\smash{{\SetFigFont{10}{12.0}{\rmdefault}{\mddefault}{\updefault}{\color[rgb]{0,0,0}$x$}%
}}}}
\put(9526,-4861){\makebox(0,0)[lb]{\smash{{\SetFigFont{10}{12.0}{\rmdefault}{\mddefault}{\updefault}{\color[rgb]{0,0,0}$x$}%
}}}}
\end{picture}%
\caption{When there are two crosses at times $t_1$ and $t_2$, connection is only with the interval $[t_1,t_2$ of length $X = t_2-t_1$; when there are two double bars, connection is only with the interval $[0,t_1] \cup [t_2,1]$ of length $1-X$; when there is one cross and one double bar, connection is with the whole interval $[0,1]$. The latter case is more favourable for connections.}
\label{fig two-connection}
\efig

As before, any link from $\rho$ to a sibling $x'$ of $x$, or from $x$
to a child $y$, is a monolink.  Links that fall in $\L_\rho\cup\L_x$
have a chance of connecting $(\rho,0)$ to generation $m$, the others do
not. 
There are $d$ choices of $x$ and the probability of exactly two links
from $\rho$ to $x$ is $\tfrac12\b^2e^{-\b}$.  Conditioning on this as
well as the lengths $|\L_\rho|$ and $|\L_x|$, and considering the
probabilities for the remaining monolinks to connect $(\rho,0)$ to
generation $m$, 
 one obtains
\be\label{PA2}
\PP(B^m_{(\rho,0)}\cap A_2)=
\tfrac12d\b^2e^{-\b}\EE\big[
\big(e^{-\b}(1+\b\zeta_{m-1}|\L_\rho|+\b (1-|\L_\rho|))\big)^{d-1}
\big(e^{-\b}(1+\b\zeta_{m-2}|\L_x|+\b (1-|\L_x|))\big)^{d}
\big],
\ee
where the expectation is over the lengths $|\L_\rho|$ and
$|\L_x|$.  As noted above, we have
\be\label{X-probs}
\begin{array}{ll}
|\L_\rho|=1-|\L_x|=X, & \mbox{with probability } u^2,\\
|\L_\rho|=|\L_x|=X, & \mbox{with probability } (1-u)^2,\\
|\L_\rho|=|\L_x|=1, & \mbox{with probability } 2u(1-u).\\
\end{array}
\ee
We now use the inequalities
\be
\zeta_{m-1}\leq 1-\tilde\s_{m-1},\qquad
\zeta_{m-2}\leq 1-\tilde\s_{m-1}
\ee
to obtain from \eqref{PA1} that
\be\label{PA1-bd}
\PP(B^m_{(\rho,0)}\cap A_1)\leq
\big(e^{-\b}(1+\b-\tilde\s_{m-1}\beta)\big)^d
\ee
and from \eqref{PA2} that
\be\label{PA2-bd}
\PP(B^m_{(\rho,0)}\cap A_2)\leq 
\tfrac12d\b^2e^{-\b}\EE\big[
\big(e^{-\b}(1+\b-\tilde\s_{m-1}\b|\L_\rho|)\big)^{d-1}
\big(e^{-\b}(1+\b-\tilde\s_{m-1}\b|\L_x|)\big)^{d}
\big].
\ee
In light of~\eqref{PA1-bd},~\eqref{PA2-bd} and~\eqref{PA12},
we will proceed by providing estimates for terms of the form
\be
\big(e^{-\b}(1+\b-\s x\b)\big)^{d},\qquad
\ee
for $\s=O(d^{-1})$ and constant $x\in[0,1]$.
Since $\b=\tfrac1d+\tfrac\alpha{d^2}$, the following are easy to verify:
\be
\begin{split}
e^{-\b}&=1-\tfrac1d+\tfrac1{d^2}(\sfrac12-\a)+
\tfrac1{d^3}(\a-\sfrac16)+O(d^{-4}),\\
1+\b-\s x\b&=1+\tfrac1d+\tfrac1{d^2}(\a-x\s d)
-\tfrac1{d^3}(\a x\s d).
\end{split}
\ee
Here and in what follows the $O(\cdot)$ may depend on $A$ (our
absolute bound on $\alpha$) but is uniform in the other parameters.
Hence
\be
e^{-\b}(1+\b-\s x\b)=
1+\tfrac1{d^2}(-\sfrac12-x\s d)+
\tfrac1{d^3}(\sfrac13-\a+x\s d-\a x\s d)+O(d^{-4}).
\ee
Combining this with 
\be
\big(1+\tfrac{a}{n^2}+\tfrac{b}{n^3}+O(n^{-4})\big)^n
=1+\tfrac an+\tfrac1{n^2}(b+\sfrac{a^2}2)+O(n^{-3})
\ee
we see that
\be\label{exp-1}
\big(e^{-\b}(1+\b-\s x \b)\big)^d=1-
\tfrac1d(\sfrac12+x\s d)+\tfrac1{d^2}\big(\sfrac13-\a+x\s d-\a x \s d
    +\tfrac12(\sfrac12+x\s d)^2\big)+O(d^{-3}).
\ee
Applying this to~\eqref{PA1-bd} and~\eqref{PA12} we obtain
\be\label{PA1-diff}
\PP(A_1)-\PP(B^m_{(\rho,0)}\cap A_1)\geq
\tilde\s_{m-1}-\tfrac{\tilde\s_{m-1}}{d}(\sfrac32-\a)-
\tfrac12\tilde\s_{m-1}^2+ O(d^{-3}).
\ee
Now consider the case of $\PP(A_2)$ and~\eqref{PA2-bd}.
Since 
\be
\tfrac12 d\b^2e^{-\b}=\tfrac1{2d}+O(d^{-2})
\ee
it suffices in this case to use~\eqref{exp-1} to order $\tfrac1d$.  We
may also replace the $d-1$ in the exponent by $d$.
We obtain that
\be
\begin{split}
&\PP(A_2)-\PP(B^m_{(\rho,0)}\cap A_2)\\
&\geq\big(\tfrac1{2d}+O(d^{-2})\big)
\EE\big[\big(1-\tfrac1d\big)-
   \big(1-\tfrac1d(\sfrac12+|\L_\rho|\tilde\s_{m-1}d)\big)
   \big(1-\tfrac1d(\sfrac12+|\L_x|\tilde\s_{m-1}d)\big)\big]+O(d^{-3})\\
&=\tfrac{\tilde \s_{m-1}}{2d}\EE[|\L_\rho|+|\L_x|]+O(d^{-3})\\
&=\tfrac{\tilde \s_{m-1}}{2d}(u^2+\tfrac43(1-u)^2+4u(1-u))+O(d^{-3})\\
&=\tfrac{\tilde \s_{m-1}}{d}(\tfrac12+u(1-u)+\tfrac16(1-u)^2)+O(d^{-3}).
\end{split}
\ee
Here we used~\eqref{X-probs} and $\bbE(X) = \frac23$.  
Adding this to~\eqref{PA1-diff} and
substituting in~\eqref{zeta-ub}, we obtain the claim.
\end{proof}

\subsection{Absence of long loops}

Interestingly, the absence of large loops for $\a<\bar\alpha(u)$ seems
harder to establish than their occurrence for $\a>\bar\alpha(u)$.
Intuitively, this is because for part of the range of $\a$ that we
consider (namely, for $\a>\sfrac12$)
the percolation-tree (obtained by keeping only edges carrying at least
one link) is infinite with positive probability, yet we
still need to show that the loops are always blocked.

We will use the
notations $p_0,p_1,p_2,\dots$ and $p_{\geq2},p_{\geq3},\dots$ for the
probabilities for a Poisson($\b$) random variable.
We also use the shorthand 
\be
q=p_{\geq3}=\sum_{j=3}^\oo e^{-\b} \tfrac{\b^j}{j!}=\tfrac{1}{6d^3}+O(d^{-4}),
\ee
and define for $m\geq3$
\be
\check\s_{m-1}=\sum_{\ell=3}^m (dq)^{\ell-3} \s_{m-\ell}.
\ee
Here and in what follows the $O(\cdot)$ may depend on $A$ (our
absolute bound on $\alpha$) but is uniform in the other parameters.
In this section we prove:

\begin{proposition}\label{small-prop}
For all $m\geq3$ we have
\[
\s_m\leq \check\s_{m-1}
\big(1+\tfrac{\a-\bar\alpha(u)}{d}+O(d^{-2})\big),
\]
where the $O(d^{-2})$ is uniform in $m$ (but depends on $A$).
\end{proposition}

The proposition implies the remaining part of Proposition~\ref{betacrit-prop}: 
\begin{proof}[Proof of Proposition~\ref{betacrit-prop}, part (a)]
Suppose $\a-\bar\alpha(u)\leq-2\eps<0$.  For $d$ large enough we have
$dq\leq 1/d^2$ and, by Prop.~\ref{small-prop},  
that 
\be\label{bound}
\s_m\leq (1-\tfrac\eps d)\check\s_{m-1}
\leq (1-\tfrac\eps d) \sum_{\ell=3}^m \s_{m-\ell} 
\big(\tfrac1{d^2}\big)^{\ell-3},
\ee
for all $m\geq3$.
We show, by induction over $m$, that
if $d$ is large enough, then there are constants $C=C(d)>0$
and $\s=\s(d)\in(0,1)$ such that 
\be\label{exp-ind}
\s_k\leq C\s^k \mbox{ for all } k\geq0.
\ee
This clearly implies the result. 
We choose $\s=1-\tfrac{\eps}{6d}$, and by choosing $C$ appropriately 
we can assume that~\eqref{exp-ind} holds for $k=0,1,2$. 
Suppose that it holds for $k\leq m-1$ for some $m\geq3$.
Then by \eqref{bound}
\be
\s_m\leq C \s^m\, (1-\tfrac\eps d) \big(\tfrac1\s\big)^3
\sum_{l=3}^m \big(\tfrac1{\s d^2}\big)^{\ell-3}
\leq C \s^m\, (1-\tfrac\eps d) \big(\tfrac1\s\big)^3
\tfrac1{1-1/{\s d^2}}.
\ee
But here the factor
\be
(1-\tfrac\eps d) \big(\tfrac1\s\big)^3
\tfrac1{1-1/{\s d^2}} = 
(1-\tfrac\eps d) (1+3\tfrac\eps {6d} + O(d^{-2}))
(1+O(d^{-2}))=
1-\tfrac\eps{2d}+O(d^{-2})\leq1,
\ee
provided $d$ is large enough.  Hence~\eqref{exp-ind} follows for
$k=m$, as required.
\end{proof}

\begin{lemma}
\label{lem small-ineq}
Assume that
\[
\PP(B^m_{(\rho,0)}\cap (A_1\cup A_2)^c)
\geq \PP( (A_1\cup A_2)^c)[1-c\check\s_{m-1}]
\]
for some constant $c>0$. Then the bound of Proposition~\ref{small-prop} holds true.
\end{lemma}

\begin{proof}
We note that, by~\eqref{PA12}, \eqref{PA1} and~\eqref{PA2}, we have that
\be\label{PA12-ineq}
\begin{split}
\PP(B^m_{(\rho,0)}\cap A_1)&\geq
\PP(A_1)\big(1-\s_{m-1}\big(1+\tfrac{\a-1}{d}+O(d^{-2})\big)\big), \\ 
\PP(B^m_{(\rho,0)}\cap A_2)&\geq
\PP(A_2)\big(1-\s_{m-2}\big(1+2u(1-u)+\tfrac13(1-u)^2+O(d^{-1})\big)\big). 
\end{split}
\ee
This uses the inequalities $\s_{m-1}\leq \s_{m-2}$ and 
$(1-x)^n\geq 1-nx$ for $x\in[0,1]$ and $n\geq1$, as well as the
asymptotics  
\be
\frac{\b}{1+\b}=\frac1d+\frac{\a-1}{d^2}+O(d^{-3}).
\ee
We have $\check\s_{m-1}\geq\s_{m-2}\geq\s_{m-1}$
and $\PP(A_1)=1-\tfrac1{2d}+O(d^{-2})$
and $\PP(A_2)=\tfrac1{2d}+O(d^{-2})$.
Together with the assumption of the lemma, we have,
using \eqref{zeta-split-eq},
\be
\begin{split}
\sigma_m &\leq 1 - \bbP(A_1) \, \bigl[ 1 - \sigma_{m-1} (1 + \tfrac{\alpha-1}d + O(d^{-2}) ) \bigr] \\
&\quad - \bbP(A_2) \, \bigl[ 
1 - \sigma_{m-2} (1 + 2u(1-u) +\tfrac13(1-u)^2+ O(d^{-1}) ) \bigr] 
- \bbP((A_1 \cup A_2)^{\rm c}) \, [1 - c\check\sigma_{m-1}] \\
&= \sigma_{m-1} \, \bbP(A_1) \, \bigl[ 1 + \tfrac{\alpha-1}d +
O(d^{-2}) \bigr] 
+ \sigma_{m-2} \, \bbP(A_2) \, \bigl[ 1 + 2u(1-u) +
\tfrac13(1-u)^2+ O(d^{-1}) \bigr] \\
&\quad + c \check\sigma_{m-1} \bbP((A_1 \cup A_2)^{\rm c})  \\
&\leq \check\sigma_{m-1} \, \bigl[ 1 + \tfrac{\alpha-1}d - \tfrac1{2d}
+ O(d^{-2}) 
+ \tfrac1{2d} + \tfrac{u(1-u) +\tfrac16(1-u)^2}{d} + O(d^{-2}) \bigr].
\end{split}
\ee
This is indeed the upper bound of Proposition~\ref{small-prop}.
\end{proof}

The rest of this section will be devoted to the proof of the assumption of Lemma~\ref{lem small-ineq}.

We write $(A_1\cup A_2)^c$ as a union 
\be
(A_1\cup A_2)^c = \bigcup_{k=1}^d (A'_k\cup A''_k),
\ee
of the disjoint events
\begin{itemize}[leftmargin=*]
\item $A_1'$: that $\rho$ has exactly one child with $\geq3$ links and
  all other children of $\rho$ have 0 or 1 links;
\item $A_k'$ for $k\geq2$: that $\rho$ has exactly $k$ children with
  $\geq2$ links;
\item $A_k''$ for $k\geq1$: that $\rho$ has exactly one child $x$ with exactly 2 links,   all other children of $\rho$ have 0 or 1 links, and
  $x$ has exactly $k$ children with $\geq2$ links.
\end{itemize}
The following bounds will be useful later:
\begin{lemma}\label{PAs-lem}
For $d$ large enough we have
\[
\sum_{k=1}^d k \PP(A'_k)\leq 2 \sum_{k=1}^d  \PP(A'_k),
\quad\mbox{ and }\quad
\sum_{k=1}^d k \PP(A''_k)\leq 2 \sum_{k=1}^d  \PP(A''_k).
\]
\end{lemma}
\begin{proof}
We start with the $A''_k$:s, which is actually the simpler case.  For
convenience, we write $A''_0$ for the event that $\rho$ has exactly one
child $x$ with exactly 2 links, and that the other children  of $\rho$ have
0 or 1 links.  Then
\be
\sum_{k=1}^d k\PP(A''_k)=
\PP(A''_0) \sum_{k=1}^d k\PP(A''_k\mid A''_0)=
\PP(A''_0) dp_{\geq2},
\ee
since the last sum is the expected number of children of $x$ with two links or more.  Similarly
\be
\sum_{k=1}^d \PP(A''_k)=
\PP(A''_0) \sum_{k=1}^d \PP(A''_k\mid A''_0)=
\PP(A''_0) \; (1-(1-p_{\geq2})^d ).
\ee
It is easy to deduce that
\be
\frac{\sum_{k=1}^d k\PP(A''_k)}
{\sum_{k=1}^d \PP(A''_k)}\to 1,
\mbox{ as } d\to\oo,
\ee
which gives the claim for the $A''_k$.
For the $A'_k$ a straightforward but tedious calculation gives that
\be
\sum_{k=1}^d k\PP(A'_k)=
dp_{\geq2}-dp_2(1-p_{\geq2})^{d-1}=\frac{5/12}{d^2}+O(d^{-3}),
\ee
and 
\be
\sum_{k=1}^d \PP(A'_k)=
1-(1-p_{\geq2})^d-dp_2(1-p_{\geq2})^{d-1}=\frac{7/24}{d^2}+O(d^{-3}).
\ee
Hence
\be
\frac{\sum_{k=1}^d k\PP(A'_k)}
{\sum_{k=1}^d \PP(A'_k)}\to \frac{10}{7}<2,
\mbox{ as } d\to\oo,
\ee
which gives the claim for the $A'_k$.
\end{proof}

The main idea in establishing the assumption of Lemma~\ref{lem
small-ineq} is to use a certain random subtree $\check T$ of
$T^{(m)}$, which we think of as the ``complex component of $\rho$''.  
We will use $\check T$ to avoid dealing explicitly with the
possibility that the loop of $(\rho,0)$ propagates across edges
carrying $\geq3$ links.  Since edges carrying $\geq3$ links are rare,
the connected component of $\rho$ in the subtree spanned by such edges
is typically small.  This subtree is bounded by edges carrying 0, 1 or
2 links each, and we will use estimates on the probability that a loop
is blocked after traversing such an edge.  It will be convenient to
define $\check T$ slightly differently than as the subtree spanned by
edges with $\geq3$ links, since we want the event $(A_1\cup A_2)^c$ be
be $\check T$-measurable.

In order to define $\check T$, it helps to think that it consists of
``bulk sites'' and ``end sites''. The root $\rho$ is a bulk site by
definition. Assume that the tree has been defined up to level $k$, and
let $x$ be a bulk site at level $k$. An offspring $y$ is
\begin{itemize}
\item[(a)] a bulk site if the number of links $n_{xy}$
on the edge $x$ is equal to 3,4,...;
\item[(b)] a bulk site if $x=\rho$, $n_{xy}=2$, and all
siblings $z$ of $y$ satisfy $n_{xz} \in \{0,1\}$;
\item[(c)] an end site if $n_{xy} \in \{0,1,2\}$, unless there is situation (b).
\end{itemize}
Note that the event $(A_1\cup A_2)^c$ is measurable with respect to
$\check T$.

We write
$\check\om$ for the configuration of crosses and double-bars
within $\check T$.   For $j=1,2$ and $1\leq\ell\leq m-1$
we write $\cE^{(j)}_\ell$ for the
set of leaves (end sites)
of $\check T$ at distance $\ell$ from $\rho$ and with
$j$ incoming links.  If 
$x\in\cE^{(1)}_\ell$ we write $t(x)$ for the time-coordinate of the
incoming link, and if $y\in\cE^{(2)}_\ell$ we write $t_1(y)$ and
$t_2(y)$ for the time-coordinates of the two incoming links.
We also let $\cE_\ell=\cE_\ell^{(1)}\cup\cE_\ell^{(2)}$
($1\leq\ell\leq m-1$) and we let $\cE_m$ be the set of vertices of
$\check T$ at distance $m$ from $\rho$.

For $y\in T^{(m)}$ we let $T_y$ be the subtree rooted at $y$,
consisting of $y$ and all its descendants in $T^{(m)}$.
For a sub-tree $T'$ of $T^{(m)}$ we write $\Om(T')$ for the set of
configurations of crosses and double-bars in $T'$.  In particular,
$\Om(T_y)$ is the set of configurations in the subtree rooted at $y$.   
We write $B^k_{(y,t)}\se \Om(T_y)$ for the set of configurations in
$T_y$ such that the loop of $(y,t)$ visits no vertex $z\in T_y$ at
distance $k$ from $y$ (note that we do not consider any incoming links
to $y$ from its parent).  
And we write $B^m_{(\rho,0)}(y)\se\Om(T^{(m)})$ for the event
that  the loop of $(\rho,0)$ visits no vertex $z\in T_y$ at
distance $m$ from $\rho$, i.e.\ the loop does not reach distance $m$ in the
subtree rooted at $y$.

The next lemma concerns the probability of blocking a loop at a vertex
$y$ when there are two links between $y$ and its parent.

\begin{lemma}\label{2-block-lem}
Let $y$ be a vertex of $T^{(m)}$ at distance $\ell$ from $\rho$, let
$0<t_1<t_2<1$, let $\om'\in\Om(T^{(m)}\sm T_y)$ be arbitrary, and let
$\om''\in B^{m-\ell}_{(y,t_1)}\cap B^{m-\ell}_{(y,t_2)}$.  Consider a
configuration $\om\in\Om(T^{(m)})$ whose restriction to $T^{(m)}\sm
T_y$ (respectively, $T_y$) is $\om'$ (respectively, $\om''$) and in
addition has exactly two links to $y$ from its parent, at times $t_1$ and $t_2$.
Then $\om\in B^m_{(\rho,0)}(y)$.
\end{lemma}

This lemma is useful since the event 
$B^{m-\ell}_{(y,t_1)}\cap B^{m-\ell}_{(y,t_2)}$ is defined entirely in
the subgraph $T_y$, which is disjoint from $T^{(m)}\sm T_y$, and due
to the bound 
\be\label{2-sigma-ineq}
\PP(B^{m-\ell}_{(y,t_1)}\cap B^{m-\ell}_{(y,t_2)})\geq
1-2\s_{m-\ell}.
\ee
\begin{proof}
Write $x$ for the parent of $y$.  In $\om'$ the points $(x,t_1)$ and
$(x,t_2)$ belong to some loops $L_1'$, $L_2'$, where possibly
$L_1'=L_2'$.  Similarly, in $\om''$ the points $(y,t_1)$ and
$(y,t_2)$ belong to some loops $L_1''$, $L_2''$, possibly equal.
Note that neither $L_1''$ nor $L_2''$ reaches distance $m-\ell$ from
$y$ in $T_y$.

We can form $\om$ by starting with $\om'\cup\om''$, and putting in
the links $(xy,t_1)$ and $(xy,t_2)$ one at a time.  When putting in
$(xy,t_1)$ we necessarily merge $L_1'$ and $L_1''$, since they were
disjoint before.  When we then put in $(xy,t_2)$ we either cause
another merge, involving $L_2''$, or we cause a loop to split.  In
either case, no loop of $T^{(m)}\sm T_y$ ever merges with a loop  
which reaches distance $m$ from $\rho$ in $T_y$.
\end{proof}

Note that, writing $\check\PP(\cdot)$ for 
$\PP(\cdot\mid \check T, \check\om)$, 
\be
\begin{split}
\PP(B^m_{(\rho,0)}\cap (A_1\cup A_2)^c)&=
\EE[\one_{ (A_1\cup A_2)^c}\check\PP(B^m_{(\rho,0)})]\\
&=\EE\Big[\one_{ (A_1\cup A_2)^c}
   \check\PP\Big(\bigcap_{\ell=1}^m
           \bigcap_{y\in\cE_\ell} B^m_{(\rho,0)}(y)\Big)\Big].
\end{split}
\ee
But by Lemma~\ref{2-block-lem} and~\eqref{2-sigma-ineq} we have
\be
\begin{split}
\check\PP\Big(\bigcap_{\ell=1}^{m}
           \bigcap_{y\in\cE_\ell} B^m_{(\rho,0)}(y)\Big)&\geq
\check\PP\Big(\bigcap_{\ell=1}^{m-1}
           \bigcap_{x\in\cE^{(1)}_\ell} B^{m-\ell}_{(x,t(x))}
           \bigcap_{y\in\cE^{(2)}_\ell} \big(B^{m-\ell}_{(y,t_1(y))}
                        \cap B^{m-\ell}_{(y,t_2(y))}\big)
                        \Big) \one\{\cE_m=\es\}\\
&=\prod_{\ell=1}^{m-1} \prod_{x\in\cE^{(1)}_\ell} 
                \check\PP(B^{m-\ell}_{(x,t(x))})
           \prod_{y\in\cE^{(2)}_\ell} \check\PP\big(B^{m-\ell}_{(y,t_1(y))}
                        \cap B^{m-\ell}_{(y,t_2(y))}\big) \one\{\cE_m=\es\}\\
&\geq \prod_{\ell=1}^{m-1} (1-\s_{m-\ell})^{|\cE^{(1)}_\ell|}
               (1-(2\s_{m-\ell})\wedge 1)^{|\cE^{(2)}_\ell|}\one\{\cE_m=\es\}\\
&\geq 1-2\sum_{\ell=1}^m \s_{m-\ell} |\cE_\ell|.
\end{split}
\ee
(Here $\s_0=1$, and the last line
is negative when $\cE_m \neq \emptyset$.)   Hence 
\be
\PP(B^m_{(\rho,0)}\cap (A_1\cup A_2)^c)\geq
\EE\Big[ \one_{ (A_1\cup A_2)^c}
\Big(1-2\sum_{\ell=1}^m \s_{m-\ell} |\cE_\ell|\Big)\Big],
\ee
and the assumption of Lemma ~\ref{lem small-ineq} follows if we show that
\be\label{small-ineq-2}
\sum_{\ell=1}^m \s_{m-\ell} 
\EE\big[\one_{ (A_1\cup A_2)^c}|\cE_\ell|\big]\leq
48 \PP((A_1\cup A_2)^c)\check\s_{m-1}.
\ee 
The following will let us establish~\eqref{small-ineq-2}
(and hence Proposition~\ref{small-prop}):

\begin{lemma}\label{EAs-lem}
For $d$ large enough, $k\geq1$,  $m\geq3$, 
and $1\leq\ell\leq m$, we have
\[
\EE\big[\one_{A'_k}|\cE_\ell|\big]\leq 4k\PP(A'_k)
       a'_\ell \qquad \mbox{ and } \qquad
\EE\big[\one_{A''_k}|\cE_\ell|\big]\leq 4k\PP(A''_k)
       a''_\ell,
\]
where $a'_1=1$, $a'_\ell=(dq)^{\ell-2}$ for $\ell\geq2$,
$a''_1=a''_2=1$, and $a''_\ell=(dq)^{\ell-3}$ for $\ell\geq3$.
\end{lemma}

\begin{proof}
It suffices to bound the conditional expectations
\be
\EE\big[|\cE_\ell| \left| \, A'_k \big]\right.\quad\mbox{and}\quad
\EE\big[|\cE_\ell| \left| \,A''_k \big]\right.
\ee
by the appropriate functions.  We prove the result for the $A'_k$, the
arguments for the $A''_k$ are similar.

There are several cases to consider, we start with $\ell=1$.
Given $A'_k$, the number of 1's in
generation $\ell=1$ has distribution Bin($d-k,p_1/(p_0+p_1)$), and
it follows that
\be
\EE\big[|\cE_1^{(1)}| \left| \, A'_k \big]\right.
=(d-k)\frac{p_1}{p_0+p_1}\leq \frac{p_1d}{p_0+p_1}
\ee
which is trivially bounded by $2k=2k a'_1$.  Next,
we have $\EE\big[|\cE_1^{(2)}| \left| \, A'_1 \big]\right.=0$, whereas for
$k\geq2$ the number of 2's in generation $\ell=1$ has distribution 
Bin($k,p_2/p_{\geq2}$), so that 
\be
\EE\big[|\cE_1^{(2)}| \left| \, A'_k \big]\right.
=k\frac{p_2}{p_{\geq2}}\leq k\leq 2k a'_1.
\ee

For $2\leq\ell\leq m-1$ we argue as follows.  We consider the subtree of
$\check T$ formed by edges with $\geq3$ links;
the number of 1's (respectively, 2's) 
in generation $\ell$ of $\check T$ equals
the size of generation $\ell-1$ of the subtree times an independent
Bin($d,p_1$) (respectively,  Bin($d,p_2$))
 random  variable.  Each edge with $\geq3$ links from $\rho$
is the root of a Galton--Watson tree of ($\geq3$):s;  these
Galton--Watson trees have offspring distribution Bin($d,q$), and hence
on average $(dq)^r$ descendants after $r$ steps.
For $k=1$ we get simply
\be
\EE\big[|\cE_\ell^{(j)}| \left| \, A'_1 \big]\right.
= (dp_j)(dq)^{\ell-2}\leq 2a'_\ell.
\ee
For $k\geq2$ there are Bin($k,p_{\geq3}/p_{\geq2}$) Galton--Watson
trees to consider, hence
\be
\EE\big[|\cE_\ell^{(j)}| \left| \, A'_k \big]\right.
= (k\tfrac{p_{\geq3}}{p_{\geq2}})(dp_j)(dq)^{\ell-2}\leq 2ka'_\ell.
\ee
For $\ell=m$ a similar argument gives
\be
\EE\big[|\cE_m| \left| \, A'_1 \big]\right.=
((1-p_0)d) (dq)^{m-2} \mbox{ and }
\EE\big[|\cE_m| \left| \, A'_k \big]\right.=
(k\tfrac{p_{\geq3}}{p_{\geq2}}) ((1-p_0)d) (dq)^{m-2}.
\ee
\end{proof}

\begin{proof}[Proof of Prop.~\ref{small-prop}]
As mentioned, it is enough to establish~\eqref{small-ineq-2}.  
Using Lemmas~\ref{PAs-lem} and~\ref{EAs-lem} as well as the inequalities
$a'_\ell\leq a''_\ell$ and $\s_{m-1}\leq\s_{m-2}\leq\s_{m-3}$, we see
that
\be
\begin{split}
\sum_{\ell=1}^m \s_{m-\ell} 
\EE\big[\one_{ (A_1\cup A_2)^c}|\cE_\ell|\big]&=
\sum_{\ell=1}^m \s_{m-\ell} \sum_{k=1}^d \big(
\EE\big[\one_{ A'_k}|\cE_\ell|\big]+
\EE\big[\one_{ A''_k}|\cE_\ell|\big] \big)\\
&\leq 4 \sum_{\ell=1}^m \s_{m-\ell} (a'_\ell+a''_\ell)
\sum_{k=1}^d k\big(\PP(A'_k)+\PP(A''_k)\big)\\
&\leq 16 \sum_{\ell=1}^m \s_{m-\ell} a''_\ell
\sum_{k=1}^d \big(\PP(A'_k)+\PP(A''_k)\big)\\
&=16\Big(\s_{m-1}+\s_{m-2}+
  \sum_{\ell=3}^m \s_{m-\ell}(dq)^{\ell-3}\Big) 
  \PP((A_1\cup A_2)^c)\\
&\leq 48   \PP((A_1\cup A_2)^c) \check\s_{m-1},
\end{split}
\ee
for $d$ large enough, as required.
\end{proof}

\section{Sharpness of the transition}
\label{sec sharp}

The arguments of Hammond~\cite{Ham2} can straightforwardly be adapted
to our setting.  We thus obtain the following `sharpness' result,
which shows that (in the interval $\b\in[d^{-1},d^{-1}+2d^{-2}]$)
there is a unique $\b_\crit$ such that
$\s(\beta)=\PP((\rho,0)\lra\oo)$ satisfies $\s=0$ for
$\b<\b_\crit$ and $\s>0$ for $\b>\b_\crit$:

\begin{proposition}\label{sharp-prop}
For $d$ large enough, the function $\b\mapsto\s(\b)$ is non-decreasing
on the interval $\b\in[d^{-1},d^{-1}+2d^{-2}]$.
\end{proposition}
\begin{proof}[Sketch proof]
Hammond's arguments~\cite{Ham2} are written for the case $u=1$ when
there are only crosses, but they apply (almost verbatim) to the
general case $u\in[0,1]$.  We provide here a synopsis of the proof,
for the reader's benefit.

The starting point is a formula for the derivative
$\frac{d\s_n}{d\b}$, involving the concept of `the added link' (called
the added \emph{bar} by Hammond).  In addition to the Poisson process
$\om$ of links (i.e.\ crosses and double-bars), let $\fa$ be an
independently and uniformly placed link in $T^{(n)}$, which is a cross
with probability $u$ and otherwise a double-bar.  Let $P^+$ and $P^-$
denote the following \emph{pivotality events}:
\be
P^+=\{(\rho,0)\overset{\om}{\not\lra}n,
(\rho,0)\overset{\om\cup\{\fa\}}{\lra}n\},
\quad
P^-=\{(\rho,0)\overset{\om}{\lra}n,
(\rho,0)\overset{\om\cup\{\fa\}}{\not\lra}n\}.
\ee
In words, $P^+$ is the event that $\fa$ creates a connection to level
$n$ that was not present in $\om$, and $P^-$ is the event that $\fa$
breaks a connection to level $n$.  We say that $\fa$ is on-pivotal if
$P^+$ happens and off-pivotal if $P^-$ happens.  Then we
have~\cite[Lemma~1.7]{Ham2}: 
\be\label{diff-piv}
\frac{d\s_n}{d\b}=|\cE_n|\big(\PP(P^+)-\PP(P^-)\big).
\ee
Here $\cE_n$ denotes the set of edges of $T^{(n)}$.

\begin{centering}
\bfig
\begin{picture}(0,0)%
\includegraphics{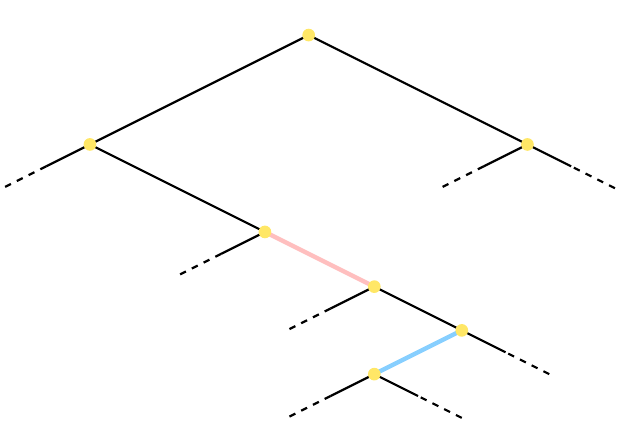}%
\end{picture}%
\setlength{\unitlength}{1381sp}%
\begingroup\makeatletter\ifx\SetFigFont\undefined%
\gdef\SetFigFont#1#2#3#4#5{%
  \reset@font\fontsize{#1}{#2pt}%
  \fontfamily{#3}\fontseries{#4}\fontshape{#5}%
  \selectfont}%
\fi\endgroup%
\begin{picture}(8466,5760)(1768,-5344)
\put(7051,-4336){\makebox(0,0)[lb]{\smash{{\SetFigFont{8}{9.6}{\rmdefault}{\mddefault}{\updefault}{\color[rgb]{0,0,0}${\mathfrak a}$}%
}}}}
\put(5851,164){\makebox(0,0)[lb]{\smash{{\SetFigFont{8}{9.6}{\rmdefault}{\mddefault}{\updefault}{\color[rgb]{0,0,0}$\rho$}%
}}}}
\put(6226,-2986){\makebox(0,0)[lb]{\smash{{\SetFigFont{8}{9.6}{\rmdefault}{\mddefault}{\updefault}{\color[rgb]{0,0,0}${\mathfrak b}$}%
}}}}
\end{picture}%
\caption{Illustration for the random link $\mathfrak a$ and the bottleneck-link $\mathfrak b$.}
\efig
\end{centering}

Hammond shows that the difference on the right-hand-side
of~\eqref{diff-piv} is positive on the interval in $\b$ considered
(when $d$ is large enough).  The result then follows by letting
$n\to\oo$.  To show that $\PP(P^+)-\PP(P^-)\geq0$, Hammond introduces
the following events.  Firstly, the \emph{crossing-event} $C$ that
the loop $\cL_{(\rho,0)}(\om)$ of $(\rho,0)$ in $\om$
visits an end-point of the added link $\fa$ before reaching level
$n$.  Note that $P^\pm\se C$, since if $C$ does not happen then the
added link has no effect on whether or not $\cL_{(\rho,0)}$ reaches
level $n$.  Secondly, the \emph{bottleneck-event} $B$ that some edge
of $T^{(n)}$ on the (unique) path from $\rho$ to $\fa$ supports only
one link.  On the event $B$, let the bottleneck-link $\fb$ be the
furthest such link from $\rho$.  And thirdly, the
\emph{no-escape-event} $N\se B$ that the loop
$\cL_{(\rho,0)}(\om\sm\fb)$ of $(\rho,0)$ in $\om\sm\fb$ does not
reach level $n$.

Note that $P^\pm$ can be written as a disjoint union
\be
P^\pm=(P^\pm\cap C\cap B^c)\cup
(P^\pm\cap C\cap B\cap N).
\ee
Indeed, one only needs to check that $C \cap B \se N$, that is, if $C$ happens and there is a
bottle-neck, then the no-escape-event happens.  But if $C$ happens
and $\fb$ is a bottle-neck, then in $\om\sm\fb$ the loop
$\cL_{(\rho,0)}$ cannot reach level $n$, because if it did then it
would reach level $n$ in both $\om$ and $\om\cup\fa$ also, since $\fb$
is a monolink (Proposition~\ref{monolink-prop}).  

Hence it suffices to
provide lower bounds on the differences
\be
\begin{split}
\d_1&=\PP(P^+\cap C\cap B^c)-\PP(P^-\cap C\cap B^c),\\
\d_2&=\PP(P^+\cap C\cap B\cap N)-\PP(P^-\cap C\cap B\cap N).
\end{split}
\ee
It is easy to give a lower bound on the first term in $\d_1$.  Indeed,
suppose the following happen: (i) in $\om$ there is no link adjacent
to $\rho$, (ii) $\fa$ is adjacent to $\rho$, (iii) the other endpoint
of $\fa$ is connected by a loop to level $n$.  Then $P^+\cap C\cap B^c$
happens.  It follows that
\be\label{Pplus-lb}
\PP(P^+\cap C\cap B^c)\geq (e^{-\b})^d \frac{d}{|\cE_n|}\s_{n-1}.
\ee
It turns out that the second term in $\d_1$ satisfies
\be\label{Pminus-ub}
\PP(P^-\cap C\cap B^c)\leq c \frac{\s_{n-1}}{|\cE_n|},
\ee
for some constant $c$ independent of $d$.  The detailed argument for
this is more involved, see~\cite[Lemma~4.5]{Ham2}, but no changes are
required compared to Hammond's original argument.  Very briefly, the
reason that one gets a constant factor $c$ rather than a factor which
grows with $d$ as in~\eqref{Pplus-lb} is as follows.  If $P^-$
happens, then necessarily the edge supporting $\fa$ also supports some
link of $\om$:  if it did not then adding $\fa$ would necessarily
merge two loops, thereby preserving any connections to level $n$.  If
also $B^c$ happens, i.e.\ there is no bottleneck, then necessarily
$\fa\in\cM\cup\cS$ where $\cM$ is the connected cluster of $\rho$
consisting of edges which support $\geq2$ links in $\om$, and $\cS$ is
the set of edges that are  adjacent to an edge of $\cM$ and support
exactly one link in $\om$.  Now $\cM$ is a very sub-critical
Galton--Watson tree, and is therefore of at most constant (expected)
size, and $\cS$ is an approximately constant (Bin($d,\b e^{-\b}$))
multiple of the number of leaves of $\cM$, and is thus also small.
Hence there is an approximately constant number of locations for $\fa$
which are consistent with the event $P^-\cap C\cap B^c$, giving the
factor $c/|\cE_n|$.  The factor
$\s_{n-1}$ appears in~\eqref{Pminus-ub} since some link of $\cS$ is
connected to level $n$.  

Putting together~\eqref{Pplus-lb} and~\eqref{Pminus-ub} we obtain
that, for $d$ large enough,  
\be\label{d1-lb}
\d_1(n)\geq \tfrac{d}{2} (e^{-\b})^d \frac{\s_{n-1}}{|\cE_n|}.
\ee
Now consider the other term $\d_2(n)$, where the bottleneck- and
no-escape-events $B$ and $N$ happen.  Since $N$ happens, any
connections to level $n$ must occur in the subtree rooted at the
bottleneck edge $\fb$, which is some (random) distance $n'\leq n$ from
level $n$.  Since $\fb$ was defined as the \emph{furthest} bottleneck
from $\rho$, there is no bottleneck in this subtree.
We thus essentially have that $\d_2(n)=\d_1(n')$, so we
can use the bounds on $\d_1$ that were already established.  The only
$n$-dependence in those bounds was in the factors $\s_{n-1}/|\cE_n|$.
It follows that for large enough $d$ we certainly have $\d_2(n)\geq0$.
Together with~\eqref{d1-lb} and~\eqref{diff-piv} this gives
\be
\frac{d\s_{n-1}}{d\b}\geq  \tfrac{d}{2} (e^{-\b})^d \s_{n-1}\geq 0,
\ee
which as explained gives the result.
\end{proof}

\smallskip
\noindent
{\bf Acknowledgments:}
We are grateful to Alan Hammond for suggesting
the present study, and for several useful comments.

This work was mainly carried out while JB was at the
University of Copenhagen (KU), in particular during
a visit of DU at that time.  It is a pleasure to thank J.\ P.\ Solovej and the KU for kind 
hospitality.

The research of JB is supported by Vetenskapsr{\aa}det grant 2015-05195.
We also thank The Leverhulme Trust for its support through the International Network `Laplacians, Random Walks, Quantum Spin Systems'.

{
\renewcommand{\refname}{\small References}
\bibliographystyle{symposium}

}

\end{document}